\documentclass[10pt]{amsart}
\usepackage{amsmath,amssymb,amsthm,graphicx,a4wide}

\newtheorem{theorem}{Theorem}    

\newtheorem{claim}{Claim} 
\newtheorem{proposition}{Proposition} 
\newtheorem{conjecture}{Conjecture} 
\newcommand{\Z}{\mathbb{Z}}

\title{Framing functions and strengthened version of Dehn's lemma}
\author{Tetsuya Ito}
\address{Research Institute for Mathematical Sciences, Kyoto university
Kyoto, 606-8502, Japan}
\email{tetitoh@kurims.kyoto-u.ac.jp}
\urladdr{http://www.kurims.kyoto-u.ac.jp/~tetitoh/}
\thanks{T.I. was partially supported by the Grant-in-Aid for Research Activity start-up, Grant Number 25887030.}

\begin{document}

\begin{abstract}
We give a lower estimate of the framing function of knots, and prove a strengthened version of Dehn's lemma conjectured by Greene-Wiest.
\end{abstract}
\maketitle

Let $K:S^{1} \rightarrow S^{3}$ be an oriented knot in $S^{3}$.
A {\em compressing disc} of $K$ is a smooth map $D:D^{2} \rightarrow S^{3}$ such that $D|_{\partial D^{2}} = K$ and that $D|_{\textrm{Int} D^{2}}$ is transverse to $K$. The intersection points of $D$ and $K$ are called {\em holes}.
In \cite{gw}, Greene and Wiest defined the {\em framing function} $n_{K}: \Z \rightarrow \Z_{\geq 0}$ of a knot $K$ by 
\[ n_{K}(k)= \min \{ \# D \cap K \: | \: D \textrm{ is a compressing disc with } i(D,K) = k\}. \]
Here $i(D,K)$ denotes the algebraic intersection number of $D$ and $K$.

The aim of this paper is to show the following theorem conjectured in \cite{gw}.
\begin{theorem}
\label{theorem:main}
If $K$ is not the unknot in $S^{3}$, then $n_{K}(0) \geq 4$. 
\end{theorem}

As is observed in \cite{gw}, this theorem can be understood as a strengthened version of Dehn's lemma. Dehn's lemma \cite{pa} says that if $K$ admits a compressing disc without holes, then $K$ is the unknot. Theorem \ref{theorem:main} says that one can weaken the hypothesis as $K$ admits a compressing disc with two holes of opposite sign. 

First of all, we give a lower bound of $n_{K}(0)$ in terms of the genus of $K$.
Although this estimate is a direct consequence of Gabai's theorem on immersed Seifert genus, it gives rise to a new insight for the framing function:
The definition and several calculations of the knot framing function in \cite{gw} suggest that $n_{K}$ is related to four-dimensional invariants of knots, like the signature and the unknotting number. The following estimate demonstrates that $n_{K}$ is also related to three-dimensional invariants of knots.

\begin{proposition}
\label{prop:genus}
Let $g(K)$ be the genus of $K$. Then $n_{K}(0) \geq 2g(K)$.
\end{proposition}
\begin{proof}
Let $D:D^{2} \rightarrow S^{3}$ be a compressing disc with $m$ positive holes and $m$ negative holes. As is observed in \cite{gw}, we may assume that $D$ is an immersion. For a pair of a positive hole $\mathsf{p}$ and a negative hole $\mathsf{n}$ of $D$ which lie adjacent on $K$, we remove small neighborhoods of $\mathsf{p}$ and $\mathsf{n}$ and attach a thin tube contained in a neighborhood of $K$, as shown in Figure \ref{fig:modification}. This gives rise to an immersion $I: \Sigma=\Sigma_{m,1} \rightarrow S^{3}$, where $\Sigma_{m,1}$ denotes the closed oriented surface of genus $m$ minus an open disc.

\begin{figure}[htbp]
\begin{center}
\includegraphics*[width=70mm]{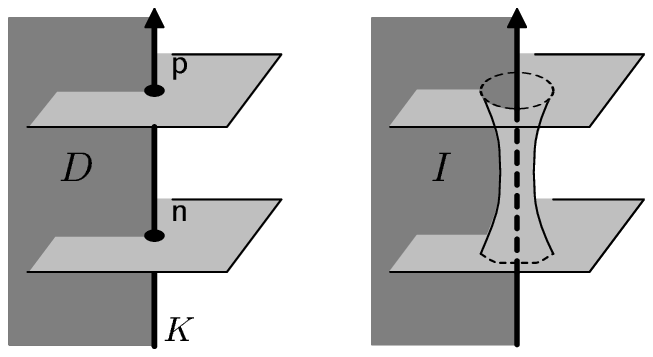}
\caption{Compressing disc $D$ and immersion $I$}
\label{fig:modification}
\end{center}
\end{figure}

The immersion $I$ is an embedding near $K=\partial \Sigma$ and $I^{-1} I(\partial \Sigma) = \partial \Sigma$, so $I$ is an immersed Seifert surface of $K$. Since the immersed Seifert genus is equal to the usual genus of $K$ \cite[Corollary 6.22]{ga}, $g(K) \leq m$. 
\end{proof}

\begin{proof}[Proof of Theorem \ref{theorem:main}]
Assume that there exists a non-trivial knot $K$ with $n_{K}(0)=2$. 
By Proposition \ref{prop:genus}, if $g(K)>1$, then $n_{K}(0)\geq 4$ so $g(K)$ must be one. Let $D$ be a compressing disc of $K$ with exactly one positive and one negative hole.

From the proof of Proposition \ref{prop:genus}, by attaching a tube to $D$ we get an immersed Seifert surface $I: \Sigma_{1,1} \rightarrow S^{3}$ of genus one.
Take a loop $\gamma \subset \Sigma_{1,1}$ so that it is homotopic to the co-core of the attached tube. Then the loop $I|_{\gamma}: S^{1} \rightarrow S^{3}-K$ represents a meridian of $K$ so $[I(\gamma)] \neq 0 \in H_{1}(S^{3}-K)$.

Let $F$ be a genus one, embedded Seifert surface of $K$ and let $M$ be the closed 3-manifold obtained by 0-framed surgery along $K$. 
By attaching one disc to $F$, we get an incompressible torus $\widehat{F}$ in $M$. Similarly by attaching one disc to the boundary of $I(\Sigma_{1,1})$, we get an immersion of a torus $\widehat{I}: T^{2} \rightarrow M$. 

We will show that $[\widehat{I}(\gamma)] = 0 \in H_{1}(M)$. Since the inclusion $S^{3}-N(K) \hookrightarrow M$ induces an isomorphism on the first homology group, this implies that $[I(\gamma)] = 0 \in H_{1}(S^{3}-K)$. This is a contradiction, which completes the proof of the theorem. We remark that $[\widehat{I}(\gamma)] = 0 \in H_{1}(M)$ if and only if the algebraic intersection of $\widehat{I}(\gamma)$ and $\widehat{F}$ is zero. In particular, if the loop $\widehat{I}(\gamma)$ is disjoint from $\widehat{F}$ or lies on $\widehat{F}$, then $[\widehat{I}(\gamma)] = 0 \in H_{1}(M)$. 

First we put the immersed torus $\widehat{I}(T^{2})$ in a nice position. This part of argument does not use the fact that we are considering tori.
We put the immersed torus $\widehat{I}(T^{2})$ so that it is transverse to $\widehat{F}$. By general position argument, we may also assume that $\widehat{F}$ does not contain branch points and triple points of $\widehat{I}(T^{2})$, so $\widehat{F} \cap \widehat{I}(T^{2})$ are immersed circles on $\widehat{F}$ with only double point singularities. Since the double points of such circles correspond to the intersections of the double point curves of $\widehat{I}(T^{2})$ and $\widehat{F}$, the preimage $\widehat{I}^{-1} (\widehat{F} \cap \widehat{I}(T^{2}))$ consists of simple closed curves of $T^{2}$.

\begin{claim}
By homotopy, we may modify $\widehat{I}(T^{2})$ so that each connected component of $\widehat{I}^{-1} (\widehat{F} \cap \widehat{I}(T^{2}))$ is an  essential simple closed curve in $T^{2}$.
\end{claim}

\begin{proof}[Proof of Claim]
By homotopy, we put $\widehat{I}(T^{2})$ so that the number of the connected components $\widehat{I}^{-1} (\widehat{F} \cap \widehat{I}(T^{2}))$ is minimal (preserving the property that the preimage $\widehat{I}^{-1} (\widehat{F} \cap \widehat{I}(T^{2}))$ consists of simple closed curves). We show that under this minimality assumption, every connected component $\alpha$ of $\widehat{I}^{-1} (\widehat{F} \cap \widehat{I}(T^{2}))$ is an essential simple closed curve on $T^{2}$. 

Assume contrary, a certain connected component $\alpha$ is null-homotopic simple closed curve that bounds a disc $\Delta_{T}$ in $T^{2}$.
Let us put $f_0=\widehat{I}|_{\alpha}:\alpha \rightarrow \widehat{F}$.
Since $\widehat{F}$ is incompressible in $M$, $f_{0}$ is a null-homotopic immersed circle in $T^{2}$ in $\widehat{F}$. Take a homotopy
$f_{s}: \alpha \rightarrow \widehat{F}$ such that $f_{1}(\alpha)$ is a simple closed curve that bounds a disc $\Delta_{F}$ in $\widehat{F}$.
Take a regular neighborhood $N(\alpha) \cong \alpha \times[-1,1] \subset T^{2}$ of $\alpha$ in $T^{2}$, and a regular neighborhood $N(\widehat{F}) \cong \widehat{F} \times [-1,1] \subset M$ of $\widehat{F}$ so that $\widehat{I}(z) = (f_{0}(x),t) \in N(\widehat{F}) \subset M$ for a point $z=(x,t) \in N(\alpha) \subset T^{2}$.

Define the homotopy $\widehat{I}_{s}: T^{2} \rightarrow M$ $(s \in [0,1])$ by
\begin{gather*}
\widehat{I}_{s}(z) = 
\begin{cases}
\widehat{I}(z), & z \not \in N(\alpha),\\
(f_{(1-t)s}(x),t), & z = (x,t) \in \alpha \times [0,1] \subset N(\alpha),\\
(f_{(1+t)s}(x),t), & z = (x,t) \in \alpha \times [-1,0] \subset N(\alpha).
\end{cases}
\end{gather*}
Then $\widehat{I}_{1}(\alpha) = f_{1}(\alpha)$ is a simple closed curve in $\widehat{F}$ bounding a disc $\Delta_{F}$ in $\widehat{F}$. 

Since $\pi_{2}(M)=0$ \cite[Corollary 8.3]{ga3}, by compressing $\Delta_{F}$ (that is, by taking a homotopy of $\widehat{I}_{1}$ so that $\widehat{I}_{1}(\Delta_{T}) = \Delta_{F}$ holds and with additional perturbation), we get a new immersion $\widehat{I'}$ such that the number of connected components of $\widehat{I'}^{-1} (\widehat{F} \cap \widehat{I'}(T^{2}))$ is strictly smaller than original immersion $\widehat{I}$. This is a contradicion.
\end{proof}

Now we utilize the fact the knot we are considering is genus one, so the surfaces we are treating are tori. The crucial point is that the fundamental group of a torus is abelian, so we have several commuting elements of $\pi_{1}(M)$. The knowledge of the structure of centralizers (abelian subgroups) in 3-manifold groups leads to restrictions for the possibilities of such commuting elements, which lead to the desired conclusion $\widehat{I}(\gamma)=0$.

Take one of the connected component $\alpha$ of $\widehat{I}^{-1} (\widehat{F} \cap \widehat{I}(T^{2}))$, and we denote the loop $\widehat{I}|_{\alpha}$ by $A$.
Since $\widehat{F}$ is incompressible in $M$, $A$ is not null-homotopic in $M$.
Take a loop $B$ on $\widehat{F}$ so that $\{A,B\}$ generates a free abelian group of rank two in $\pi_{1}(M)$. For a group $G$ and its element $x \in G$, let $Z_{G}(x) = \{ y \in G \:| \: yx=xy \}$ be the centralizer of $x$ in $G$. As an element of $\pi_{1}(M)$, $\widehat{I}|_{\gamma}$ commutes with $A = \widehat{I}|_{\alpha}$. Thus both $B$ and $\widehat{I}|_{\gamma}$ belong to $Z_{\pi_{1}(M)}(A)$.
There are three possibilities for the structure of the centralizer $Z_{\pi_{1}(M)}(A)$ \cite{j,js,fr}:
\begin{enumerate}
\item[(i)] $Z_{\pi_{1}(M)}(A) \cong \Z$.
\item[(ii)] There exist $h \in \pi_1(M)$ and an incompressible torus $\mathcal{T}$ in $M$ such that $Z_{\pi_{1}(M)}(A) = h\pi_{1}(\mathcal{T}) h^{-1}$.
\item[(iii)] There exist $h \in \pi_1(M)$ and a Seifert fibered component $\mathcal{S}$ in the geometric decomposition of $M$ such that $Z_{\pi_{1}(M)}(A) = hZ_{\pi_{1}(\mathcal{S})}(h^{-1}Ah) h^{-1}$.
\end{enumerate}

The case (i) never occurs since we assumed $A$ and $B$ generates a free abelian group of rank two. If the case (ii) occurs, then $A$ and $B$ are loops on the  $\widehat{F}$ and the loop $\widehat{I}|_{\gamma}$ is homotopic to a loop contained in $\widehat{F}$. This implies $[\widehat{I}(\gamma)] = 0 \in H_{1}(M)$.

Finally assume that the case (iii) occurs, so $\widehat{F}$ is an incompressible torus in $\mathcal{S}$. If $\widehat{F}$ is separating, then the algebraic intersection number of $\widehat{I}(\gamma)$ and $\widehat{F}$ is zero, so $[\widehat{I}(\gamma)] = 0 \in H_{1}(M)$. Hence $\widehat{F}$ must be non-separating.

Assume that $\widehat{F}$ is horizontal surface in $\mathcal{S}$. Since $\widehat{F}$ is non-separating, a connected component of the intersection of $\widehat{F}$ and the base orbifold $\mathcal{O}$ of the Seifert fibration yields a non-separating simple closed curve $C$ in $\mathcal{O}$ representing a non-trivial homology class in $\mathcal{S}$. $\mathcal{S}$ is a Seifert fibered component of the geometric decomposition of $M$, so either $M=\mathcal{S}$, or $M$ is obtained by gluing $\mathcal{S}$ and other 3-manifold $M'$ along their torus boundaries. Since boundary tori of $\mathcal{S}$ is the boundary of $\mathcal{O}$ times regular fiber, gluing other 3-manifold $M'$ does not kill the homology class represented by $C$. Therefore, $C$ also represents a non-trivial homology class in $M$. On the other hand, $C$ is a curve on $\widehat{F}$ so it must represent the trivial homology class of $M$, which is a contradiction.
 
This concludes that $\widehat{F}$ is a vertical, non-separating incompressible torus in $\mathcal{S}$, hence $\mathcal{S}$ is a Seifert fibered 3-manifold which is a torus bundle over the circle. This happens only if $K$ is a trefoil, since 0-framed surgery on a knot in $S^{3}$ yields a fibered 3-manifold if and only if the knot is fibered \cite[Corollary 8.19]{ga3}. However, $n_{\textsf{Trefoil}} (0) = 4$ \cite{gw}, this cannot happen.

\end{proof}

We close the paper by posting the following conjecture which generalizes Greene-Wiest's one.   

\begin{conjecture}
$n_{K}(0) \geq 4g(K)$.
\end{conjecture}

Our results shows that this is the case if $g(K)=1$. This is also true for the case $K$ is a torus knot \cite{gw}.

\end{document}